\documentclass[12pt]{amsart}
\usepackage{latexsym,fancyhdr,amssymb,color,amsmath,amsthm,graphicx,listings,comment}
\usepackage[section]{placeins}
\newtheorem{thm}{Theorem}
\newtheorem{lemma}{Lemma}
\newtheorem{propo}{Proposition}
\newtheorem{coro}{Corollary}
\newtheorem{conj}{Conjecture}
\let\paragraph\subsection

\title{Remarks on the Brouwer Conjecture}
\author{Oliver Knill}
\date{August 12, 2025}
\address{Department of Mathematics \\ Harvard University \\ Cambridge, MA, 02138 }


\begin{document}
\maketitle

\begin{abstract}
The Brouwer conjecture (BC) in spectral graph theory claims that
the sum of the largest k Kirchhoff eigenvalues of a graph are bounded
above by the number m of edges plus k(k+1)/2. We show that (BC) holds 
for all graphs with n vertices if n is larger or equal than 
4 times the square of the maximal vertex degree. 
We also note that (BC) for graphs implies (BC) for quivers. 
\end{abstract}

\section{In a nutshell}

\paragraph{}
The {\bf Brouwer conjecture} (BC) is the statement in spectral graph theory that 
$S_k=\sum_{j=1}^k \lambda_j \leq m+k(k+1)/2 = B_k$ for all $1 \leq k \leq n$,
where $\lambda_1 \geq \lambda_{2} \geq \cdots \lambda_n$ are the non-increasing ordered
eigenvalues of the Kirchhoff matrix of the graph and $m$ is the number of edges. 
The now 20 year old conjecture is still open despite considerable effort like
\cite{Chen2019,WangLinZhangYe,HMT2010,TorresTrevisan, Rocha2020,Cooper2021,LiGuo2022}.

\paragraph{}
We add here some remarks related to the conjecture. Some involve the sum
$D_k=\sum_{j=1}^k d_j$ of the largest $k$ vertex degrees.
In the case of quivers, graphs with possible loops and 
$r$ additional multiple connections, the constants have to be adapted to 
$B_k=m+r+k(k+1)/2$. Define also $H_k=m+r+k^2$ which recently appeared in 
\cite{Lew2025} for finite simple graphs (where $r=0$).

\paragraph{}
Here are the five major points we cover here. "Quiver" can always be replaced
with "finite simple graph" as the later is the special case of a quiver
without loops and multiple conditions. 

\begin{itemize}
\item $D_k \leq S_k \leq 2D_k$ holds for all quivers.
\item $D_k \leq B_{k-1}$ holds for all quivers.
\item (BC) holds for all quivers $G$ with $n(G)>4d_1(G)^2$ vertices.
\item If (BC) holds for simple graphs, it holds for quivers. 
\item $S_k \leq H_k$ holds for all quivers. 
\end{itemize}

\paragraph{}
We also mention that if we would know the inequality 
$S_k \leq B_k$ for all $k \leq s$ and for all graphs $G$, 
then (BC) would hold for all graphs satisfying $\lambda_1 \leq s$.
As of now, this is only known for $s=2$ and is not very useful yet.
We do not even know whether there is an example of a graph 
with $\lambda_1+\lambda_2 +\lambda_3 > m+6$. 

\section{The Brouwer conjecture}

\paragraph{}
Let $G=(V,E)$ be a {\bf finite graph} with $n$ {\bf vertices} $V$ and $m$ 
{\bf edges} $E$. Self loops and multiple connections both count as edges.
The {\bf Kirchhoff matrix} $K$ of such a quiver $G$ is defined 
as $K=F^T F$, where the $m \times n$ matrix $F$ is the 
{\bf quiver gradient} while its transpose, the 
$n \times m$ matrix $F^T$ is the {\bf quiver divergence} 
\cite{Knill2024} and $K=F^T F = {\rm div} {\rm grad}$ is the 
{\bf quiver Laplacian}. The matrix $K$ can be written as $D-A$ with diagonal 
{\bf vertex degree matrix} $D$ (in which every loop counts as $1$) and 
{\bf adjacency matrix} $A$ in which $-A_{ij}$ 
counts the number of edges between node $i$ and
node $j$. Let $\lambda_1 \geq \lambda_2 \cdots \geq \lambda_n$ denote the 
{\bf eigenvalues} of $K$, ordered in a non-increasing manner. 
The identity $K=D-A$ shows that eigenvalues of $G$ do not depend on the 
choice of the orientations on $E$, despite the fact that the orientation 
was used to define the gradient matrix $F$. We remind in an appendix a bit
of more notation about quiver calculus as the literature is not uniform
in terminology. As in \cite{Knill2024} we also add computer algebra code
for quivers.

\paragraph{}
Assume first that $G$ is a {\bf finite simple graph}, an 
undirected quiver without multiple connections and no loops. 
Let $S_k(G) = \sum_{j=1}^k \lambda_j$ be the 
{\bf cumulative descending spectral sum}. The 
{\bf Brouwer bound} is defined as $B_k(G)=m(E)+k(k+1)/2$, 
where $m(G)$ is the {\bf number of edges} of $G$. 
The {\bf conjecture of Brouwer} (BC) (see \cite{Brouwer}, 
page 53) has first been formulated if $G$ is a 
finite simple graphs, meaning that no loops nor multiple 
connections are allowed. 

\begin{conj}[BC]
If $G$ is a finite simple graph, 
then $S_k \leq B_k$ for all $1 \leq k \leq n$. 
\end{conj}

\paragraph{}
The conjecture is open. In all our experiments so far, we also 
see that also the {\bf sign-less Brouwer conjecture} (BC+) holds, 
in which the Kirchhoff Laplacian K is replaced by 
$|K|$. This is related to the {\bf connection Laplacian} $L$
which is a $n \times m$ unimodular matrix, has the property that $L -L^{-1}$  
is block diagonal with $K=|F^T F|$ and $K_1=|F F^T|$ as blocks. There has 
been considerable interest in the conjecture during the last couple of years
\cite{Chen2019,WangLinZhangYe,HMT2010,AshrafOmidiTayfee2013,TorresTrevisan,
Rocha2020,Cooper2021,LiGuo2022}.
The signless Brouwer conjecture (BC+) has appeared
first in \cite{AshrafOmidiTayfee2013}, where the cases $k=1,2$ are covered.
We see that if (BC) or (BC+) holds for finite simple graphs, it 
extends to all quivers when suitable adapted in the case of multiple connections.

\section{A sandwich} 

\paragraph{}
Let us first look at relations with 
{\bf vertex degree sequences}. Define $D_k(G)=\sum_{j=1}^k d_j$
if $d_1 \geq d_2 \geq \cdots \geq d_n$ are the sorted vertex degrees 
of the quiver. These are also 
the diagonal elements of $K$ when sorted in a non-increasing manner. 
The {\bf Schur inequality}, which relates diagonal entries and eigenvalues for
any symmetric matrix, immediately gives $S_k \geq D_k$. We proved in \cite{Knill2024},
the general upper bound $\lambda_k \leq d_k+d_{k+1}$ for all quivers
and all eigenvalues $\lambda_k$.  This especially implies:

\begin{thm}[Sandwich]
For any quiver we have $D_k \leq S_k \leq 2D_k$. 
\end{thm}

\paragraph{}
Especially for larger $n$, we see in experiments 
that the spectral sum $S_k$ is sandwiched between $D_k$ and $c D_k$ with $c$ close 
to $1$ if $k \geq 2$. The sequences $D_k,S_k$ entering the 
inequalities $D_k \leq S_k \leq 2 D_k$ are concave down
sequences, while the Brouwer bound $B_k$ 
is concave-up functions in $k$. 
Elementary is the following estimate. We suspect that it could have been one on of the
motivations to (BC) even this is so not explicitly stated as such in the literature.

\begin{propo}
$D_k \leq B_{k-1}$ for all quivers.
\end{propo}
\begin{proof} 
(i) Assume first that $G$ has no loops. 
Look at the sub-graph which is the union of the $k$ star graphs
centered at the $k$ vertices with largest vertex degree. 
Such a graph has at least $D_k-k(k-1)/2$ edges as maximally 
$k(k-1)/2$ can be double counted. This shows that $D_k-k(k-1)/2 \leq m$. \\
(ii) If we add a loop to a vertex in the graph, 
then $B_k$ increase by $1$ because $m$ increases by $1$.
The $D_k$ can only increase by either $0$ or $1$ depending on whether the loop
was at one of the $k$ largest degree vertices. \\
(iii) If we add an additional multiple connection, then $D_k$ can only grow by 
$2$ or less, while $B_{k-1}=m+r+k(k-1)/2$ grows by $2$ as both $m$ and $r$
(the redundancy) grow by $1$. We will define the redundancy more precisely 
$r$ in the next section. 
\end{proof} 

\paragraph{}
We would like to know the largest $c>1$ such that the concave down 
sequence $c D_k$ is ``tangent" to the concave up sequence $B_k$. 
We know that there exists a constant $1<b<2$ defined as the smallest $b$ 
such that $S_k \leq b D_k$. We often see that $D_k \leq S_k \leq c D_k \leq B_k$
which would hold if $b \leq c$. The difficulty of the Brouwer conjecture 
can be illustrated in reminding that even for $k=3$ the 
statement $S_3 \leq B_3$ is not yet known in general. Written out, this case $k=3$
means that the sum of the 3 largest eigenvalues and the number $m$ of edges
satisfy $\lambda_1 + \lambda_2 + \lambda_3 \leq m+6$. Even that case is open. 

\section{Extension to quivers} 

\paragraph{}
The total number $r$ of {\bf redundant edges} 
in a quiver is defined as the minimal number of edges which 
when removed produces a graph without multiple 
connections. This {\bf redundancy} $r=r(G)$ can also be
read off from the Kirchhoff matrix $K$ as
$r=\sum_{i<j} {\rm sign}(K_{ij})-K_{ij}$. 
In this more general setting, again set $S_k =\sum_{j=1}^k \lambda_j$ and define 
$B_k=m+r+k(k+1)/2$, if $m$ still counts the total
number of edges, including loops and multiple connections and where $r$ is 
the redundancy of $G$. For graphs without multiple connections, this regresses to
$B_k=m+k(k+1)/2$ which is the usual Brouwer bound. Note that adding a multiple
connection increases $B_k$ by $2$ as both $m$ and $r$ get increased. 

\paragraph{}
We will just see that if $G$ is a (BC) graph, then adding a loop or a multiple 
connection keeps the extended graph in (BC). 

\begin{thm}
If (BC) holds all finite simple graphs $G$, then (BC) holds for all quivers $G$. 
\end{thm} 
\begin{proof} 
(i) Let $v$ be the vertex which belongs to the $x$'th row in the matrix $K$. 
The {\bf Hadamard perturbation formula}
$\lambda_i' = |v_x|^2$ for a unit eigenvector
$v$ shows after integration from $t=0$ to $t=1$ that if a loop is added, the
eigenvalues can only grow. Comparing the trace shows that $S_k$ can maximally grow by $1$ 
while the Brouwer sum $B_k$ grows by exactly $1$.  \\ 
(ii) If we add an additional edge $(x,y)$,
the Hadamard perturbation formula for the deformation 
gives $\lambda_k' = (v_x-v_y)^2 \geq 0$. Looking at the trace
shows that the spectral sum $S_k$ grows maximally by $2$. 
The Brouwer sum $B_k =m+r+n(n+1)/2$ grows for each $k$ by exactly $2$ during the deformation,
because $m$ is increased by $1$ and $r$ is increased by $1$. 
\end{proof} 

\paragraph{}
{\bf Remarks.} \\
a) The same same upgrade is possible for the sign-less conjecture (BC+). 
We therefore would only need to show the sign-less conjecture for finite simple graphs and
get the sign-less conjecture for general quivers. \\
b) Instead of referring to Hadamard deformation, 
one could also invoke general principles of symmetric operators. 
If $K \leq L$ are two symmetric matrices, then
$\lambda_j(K) \leq \lambda_j(L)$. The perturbation adds 
a projection $e_j \cdot e_j^T$ in the first case and a 
multiple of a projection 
$(e_i - e_j) \cdot (e_i \cdot e_j)^T$ in the second case.

\paragraph{}
Extending the conjecture to quivers can be motivated by modeling 
{\bf Schr\"odinger cases} (in the sense of providing discrete "tight binding 
approximation" versions of $-\Delta +V$, 
where the loops model a {\bf potential values} $V(v)$ 
at a vertex $v$. It also allows for extended investigations. We can 
for example look at the least amount $l$ of loops which are additionally needed
at each vertex such that (BC) can be proven. 

\begin{propo}
If sufficiently many loops are attached at every vertex, 
then (BC) hold. 
\end{propo}
\begin{proof}
If $l$ loops are attached to each vertex, the
Kirchhoff matrix becomes $K+l I_n$ so that $S_k(G+L) = S_k(G) + l*k$ and 
$B_k(G+L)=B_k(G)+l*n$. If $l$ is large enough, then 
$S_k(G) + lk \leq B_k(G) + ln$ holds for all $k<n$. 
But for $k=n$, we have $S_n(G) \leq B_n(G)$ which is known. 
\end{proof} 

\paragraph{}
This motivates to stratify the conjecture and to ask for example, whether 
it is possible to prove the conjecture for all graphs with a generous
number of $l=1000$ loops attached at each vertex. 

\section{Large graphs}

\paragraph{}
Illustrated by the fact that we do not know
whether $S_3 \leq B_3$ in general, there appears to 
be a difficulty to establish the bound for small $k$. We have also sharp situations
for very dense graphs, as the case complete graphs $K_n$ illustrate, where $k=n-1$ is sharp. 
This suggests that we should take large enough graphs, but keep the 
edge degree bounded in order to be able to make a statement. 

\paragraph{}
Let us start with an interlacing result. If we remove an edge from the graph, the
number of vertices does not change, so that applying the Cauchy-interlace theorem appears
not possible. Removing an edge however produces a principal submatrix of
the 1-form Laplacian $K'=F F^T$ which is essentially 
isospectral to $K=F^T F$ and which is a $m \times m$ matrix if there are $m$ edges.
Let $\lambda_1 \geq \lambda_2 \cdots \lambda_n$ be the eigenvalues of $K$. 
If $H=G-e$ is a subgraph in which one of the edges has been removed, we get the
eigenvalues $\mu_1 \geq \mu_2 \cdots \mu_n$. 
We say that $\mu_l$ list is {\bf interlaced} with the $\lambda_l$ list, 
if $\lambda_1 \geq \mu_1 \geq \lambda_2 \geq \cdots \geq \mu_n=\lambda_n=0$. 

\begin{lemma}[See \cite{Godsil} theorem 13.6.2]
The eigenvalues of $K(G-e)$ interlace the eigenvalues of $K(G)$.  
\end{lemma}
\begin{proof}
Here is a new ``supersymmetry proof" 
Because $K=F^TF$ and $K'=F F^T$ are essentially isospectral, it is enough 
to show the interlacing for $K'$. But $K'(G)$ has the row and column by $e$
deleted, so that by the {\bf Cauchy interlace theorem}, the spectra of $K'(G)$
and $K'(G-e)$ are interlaced.
\end{proof} 

\paragraph{}
The general statement that the eigenvalues $\lambda_k(H) \leq \lambda_k(G)$ if $H$ 
is a subgraph of $G$ follows also from the {\bf Courant variational principle}. 
But more is true. The spectrum of a subgraph always interlaces the spectrum of the graph. 
One can ask whether the result $\lambda_j \leq d_j+d_{j+1}$ (decreasing eigenvalue
and vertex degree assumption), proven in \cite{Knill2024} can also be proven by induction
while removing edges. The answer is ``no" because removing an edge reduces some vertex degrees
rendering induction impossible. The extension of the result to quivers was necessary. 

\paragraph{}
The next theorem is a contribution to the Brouwer conjecture that
adds confidence that (BC) holds for all graphs.
Recall that $n$ is the number of vertices of the graph, while $m$ is the 
number of edges of the graph and $d_1$ is the {\bf maximal vertex degree}.

\begin{thm}
If $G$ is a connected graph for which $n \geq 4d_1^2$, then (BC) holds for $G$. 
\end{thm}

\begin{proof} 
We can first assume that $G$ is a finite simple graph, then use that we 
can upgrade the result from graphs to quivers. 
Take a spanning tree $H_0 \subset G$. This requires that $G$ is connected. 
We know (BC) holds for all trees \cite{HMT2010}. 
We will now produce a sequence of graphs 
$H_0 \subset H_1 \subset \cdots \subset H_{m-n+1}=G$,
where going from $A=H_l$ to $B=H_{l+1}$ is done by adding an edge of $G$. 
We show that if $A=H_L$ satisfies (BC), then $B=H_{l+1}$ satisfies (BC). 
Now use the above lemma which assures that if $A$ has eigenvalues 
$\mu_1 \geq \mu_2 \geq \cdots \geq \mu_n=0$
and $B$ has eigenvalues $\lambda_1 \geq \lambda_2 \geq \cdots \geq \lambda_n=0$ then 
$\lambda_1 \geq \mu_1 \geq \lambda_2 \geq \mu_2 \cdots \geq \lambda_n=\mu_n=0$.
Especially $\mu_j \geq \lambda_{j+1}$. \\
{\bf First case (i)} $k \geq 2 d_1$: this implies $k \geq \lambda_1$.
\begin{eqnarray*} 
     S_k(B) &=& \lambda_1 + (\lambda_2 + \lambda_3 + \cdots + \lambda_k) \\
            &\leq& \lambda_1 + (\mu_1  + \mu_2 + \cdots \mu_{k-1}) \\
            &\leq& \lambda_1 + m(A) + k(k-1)/2  \\
            &=   & \lambda_1 -k + m(A) + k(k+1)/2 \\
            &\leq&  m(A) + k(k+1)/2 \\
            &<&     m(B) + k(k+1)/2 = B_k(B)  \; . 
\end{eqnarray*}
{\bf Second case (ii)} $k \leq 2 d_1$: Then $k \lambda_1 \leq (2 d_1)^2$ and
\begin{eqnarray*}
      S_k(B) &=& \lambda_1+\lambda_2 + \cdots + \lambda_k \\
             &\leq&  k \lambda_1    \\
             &\leq&  (2 d_1) \lambda_1 \\
             &\leq&  (2 d_1) (2 d_1) \\
             & =&    4 d_1^2  \\
             &\leq&  n  \\
             &\leq&  m+1 \\
             &\leq&  m+k(k+1)/2= B_k(B)  \; . 
\end{eqnarray*} 
So, also $B=H_{l+1}$ satisfies (BC). We continue as such until 
$H_{m-n+1} = G$ is reached and see that $G$ satisfies (BC).
\end{proof} 

\paragraph{}
{\bf Examples.} \\
a) In the case of a {\bf cyclic graph} $C_n$, we have $m=n$ and $d_j=2$ 
and the result works for $n \geq 4*2^2=16$. 
Cyclic graphs are well known to be in (BC): the case $n=2,n=3$ is covered
with complete graphs. For $n \geq 4$, we can use that $1+2k \leq k(k+1)/2$. 
Now, $\lambda_j \leq 4 \sin(\pi j/(2n))^2$ has the explicit sum
$S_k = 1+2k-\csc(2n/\pi) \sin((1+2k)\pi/(2n))
     \leq 1+2k \leq k(k+1)/2 \leq n +k(k+1)/2 = B_k$.  \\
b) For a general {\bf triangle free graph}, the Barycentric refinement 
doubles $m$ but keeps the maximal vertex degree constant. So, 
after doing more than $\log_2(4d_1^2/m)$ refinement steps to a triangle free
graph, we are in (BC). \\
c) For a vertex {\bf regular graph} of degree $d$, we have $nd=2m$. 
Now $4d^2=4 (2m)^2/n^2=16m^2/n^2 \leq n$ if $16m^2 \leq n^3$. \\
\footnote{Alan Lew informed me that Brouwer's conjecture is known to hold for 
regular graphs, by work of Mayank \cite{Mayank2010} and Berndsen \cite{Berndsen2010} 
in their M.Sc. Theses.} \\
d) For a random graph in the Erdoes-Renyi space $E(n,p)$, the 
vertex degree expectation is about $d=pn$ so that we need $4p^2n^2<n$
or $p<1/\sqrt(2n)$ in order that the theorem can be applied for large $n$.
There are results about random graphs in \cite{Rocha2019,Cooper2021}. 
Rocha showed that (BC) holds asymptotically almost surely. \\
e) Two discrete manifolds can be connected via a {\bf connected sum}.
In the special case when doing a connected sum with a spheres 
we get a notion of homeomorphism. If we do connected sums
using manifolds with a global upper bound on the vertex degree, then 
taking a sufficiently large manifold of this type will get us into the class (BC). \\
f) For a {\bf discrete 2-manifold} we can make a soft Barycentric refinement
By taking the union of triangles and vertices as new vertices and 
connect two if one is contained in the other or if the intersection 
is an edge. This operation does not increase the maximal vertex degree
if the maximal vertex degree is already 6 or more.

\section{The snap reduction} 

\paragraph{}
We have just seen that edge removal produces interlacing eigenvalues.  
There is a snap construction within the class of quivers,
which has the effect of removing a vertex and adding loops. This operation
was used in \cite{Knill2024}. It is somehow dual to the edge removal because 
the number of edges does not change, while the number of vertices decreases 
by $1$. 

\paragraph{}
Assume $G$ is a finite simple graph or more generally a quiver
without multiple connections. If we remove a vertex from $G$ and snap 
the edges to the nearest connections so that edges morph to loops,
we end up with a quiver $H=G-v$ with $n-1$
vertices for which the total edge sum $m$ of edges remains the same 
if no loop was present at $v$ and where the total edge sum $m$ has 
decreased to $m-l$ if $l$ loops had been present at $v$.  
We call $G-v$ the {\bf snap reduction} of $G$. 

\paragraph{}
An observation in \cite{Knill2024} was that the snap reduction has the 
effect that the Kirchhoff matrix $K(G-v)$ is a principal 
sub-matrix of the Kirchhoff matrix $K(G)$ of $G$. 
The {\bf Cauchy interlace theorem} then relates the eigenvalues of 
$H=G-v$ and $G$. In the Brouwer we can use the snap reduction
if $k$ is larger than the spectral radius. 
If $k$ is smaller than the spectral radius, we have to bound $k$ terms each 
smaller or equal than the spectral radius.

\begin{lemma}[Snap lemma]
Let $G$ be a quiver without multiple connections and let $H=G-v$ be a snap reduction
of $G$. Then $B_k(H) \leq B_k(G)$ and $S_k(G) \leq B_k(G) + (\lambda_1-k)$
and especially $S_k(G) \leq B_k(G)$ for all $k \geq \lambda_1(G)$. 
Also $\lambda_1(H) \leq \lambda_1(G)$. 
\end{lemma} 

\begin{proof}
If we remove a simple connection from $v$ and snap it to a loop at 
$w$, then the edge number $m$ either stays the same or decreases by the number of 
loops at $m$. Therefore, also $B_k=m+k(k+1)/2$ stays the same or becomes smaller. 
The Kirchhoff matrix of the snapped graph is a {\bf principal sub-matrix} 
of the Kirchhoff matrix of $G$. 
This was true also if the vertex $v$ has loops.
Let $\mu_j$ denote the eigenvalues of the snapped graph $H=G-v$
Then, using $B_k(H) = B_{k-1}(H) +k$ and $B_k(H) \leq B_k(G)$ 
(following $m(H)\leq m(G)$), we have by induction and using $k \geq \lambda_1$: 
$$ S_k(G) = \sum_{j=1}^n \lambda_j \leq \lambda_1 + S_{k-1}(H) \leq \lambda_1 + 
           B_{k-1}(H) = \lambda_1 -k + B_{k}(H)  
          \leq \lambda_1 - k + B_k(G)  \leq B_k(G) \; . $$ 
\end{proof} 

\paragraph{}
Define $\mathcal{G}_{\sigma}$ as the set of quivers 
with spectral radius $\leq \sigma$.

\begin{lemma}
The class $\mathcal{G}_{\sigma}$ is invariant under snap reduction
and edge removal. 
\end{lemma}

\begin{proof} 
In both cases, the spectrum of the reduced graph is interlaced, 
the spectral radius of $H=G-v$ is smaller or equal than the 
spectral radius of $G$. So, if $G \in \mathcal{G}_{\sigma}$ then 
$H \in \mathcal{G}_{\sigma}$. 
\end{proof} 

\paragraph{}
Here is a way to enlarge graphs for which (BC) holds: if we take a graph $H$
satisfying (BC) and which has enough edges, then we can extend as long
as we do not increase the maximal vertex degree. The reason is that we 
keep an upper bound on the spectral radius. 

\begin{propo}
If (BC) holds for a connected graph $H$ with $m(H)>4d_1(H)^2$ 
and $H$ is a subgraph of $G$ with $d_1(G) \leq d_1(H)$, 
then (BC) also holds for G.
\end{propo}

\begin{proof}
If $G$ is a quiver with $m$ vertices and spectral radius $\lambda_1$.
If $m \geq \lambda_1^2$, then $G$ satisfies (BC) if $H$ satisfies $BC$.
(i) If $k \leq \sigma(G)=\lambda_1$ then
$D_k=\lambda_1 + \cdots + \lambda_k \leq k \sigma(G) 
  \leq \sigma(G)^2 \leq m \leq m+k(k+1)/2=B_k$.  \\
(ii) In the case $k \geq \sigma(G) = \lambda_1$, we use the snap lemma.
Let $G$ be a quiver with $n+1$ vertices. Take any vertex $v$ in $G$ and snap it to get a 
graph $H=G-v$. The new graph $H$ has a Kirchhoff matrix with interlaced spectrum
so that the induction step works also in the case $k \geq \sigma =\lambda_1$. The problem 
is that the snapped graph has in general less edges so that $m(H) \geq \lambda_1(H)^2$ does
not hold. 
\end{proof} 

\paragraph{}
Here is a question. A {\bf 2-manifold} is a finite simple graph such that
every unit sphere is a circular graph with $4$ or more elements. 
A {\bf 2-manifold with boundary} is a finite simple graph for which 
every unit sphere is either a circular graph with $4$ or more elements or then
a path graph of length 2 or more. An example of a 2-manifold with boundary 
is to take a boundary-less graph and remove a vertex. It would be nice to 
know whether all q-manifolds are in (BC). We so far only know that it holds
for large enough manifolds provided that the curvature stays bounded below. 

\section{Brouwer threshold} 

\paragraph{}
Define the {\bf Brouwer threshold} $s$ as the largest 
number $s$ such that the Brouwer estimate $S_k(G) \leq B_k(G)$ holds
for all $k \leq s$ for all finite simple graphs.  This threshold is a {\bf global
number} that only depends on the progress of research. The Brouwer conjecture is
equivalent to the statement that $s = \infty$. As for now (2025), we only know $s=2$. 
That the Brouwer threshold $s$ can be pushed up in the future is a reasonable 
expectation. Also the next statement is a motivation to increase the threshold. 
It does not invoke any condition in the number of edges.

\begin{coro}
If $s$ is the Brouwer threshold. For all quivers with 
$\lambda_1 \leq s$, the Brouwer conjecture holds.
\end{coro}

\begin{proof}
This also immediately follows from combining the two lemmas 
and using induction in the class $\mathcal{G}_{\sigma}$
using that once we have established Brouwer for all finite simple 
graphs of order $n$, then it holds for all quivers of order $n$. 
As for the induction assumption, the claim holds for $n=1$,
where it holds unconditionally. 
\end{proof} 

\section{Signless Laplacian}

\paragraph{}
The {\bf signless Laplacian} $|K|$ is defined as $|K|=D+A$ or 
$|K|_{ij}=|K_{ij}|$. If $K=F^T F$, then $|K|=|F|^T |F|$. The 1-form version 
is $K'=|F| |F^T|$. If we place $F$ into the lower sub-diagonal of a 
$(n+m) \times (n+m)$ 
matrix which is zero everywhere else, we call it the {\bf exterior derivative} 
$d$. Now define the {\bf Dirac matrix} $D=d+d^*$ which is a symmetric 
$(n +m) \times (n+m)$ matrix. Since $d^2=0$ the matrix $H=D^2=(d+d^*)^2
= d d^* + d^* d$ is block diagonal and has in the first block a $n \times n$
Kirchhoff matrix $K=F^T F$. The second block is the $m \times m$ matrix 
$K_1 = F F^T$. The Hodge Laplacian is $H=K \oplus K_1$. The two blocks are 
essentially isospectral, meaning that their non-zero eigenvalues agree. 

\paragraph{}
If we replace $d$ with $|d|$ and $d^*$ with $|d^*|$ the same analysis
works but now $|H| = |K| \oplus |K_1|$ and the two matrices
$|K|$ and $|K_1|$ are still essentially isospectral. (Note that the sign-less
Laplacian in higher dimensions is no more block diagonal as $d^2=0$ does not hold
any more then.)
It is known that the spectral radius of $|K|$ is larger or equal than the 
spectral radius of $K$, but there is no definite inequality between the 
smaller eigenvalues. The eigenvalues of $H$ and $|H|$ are very close in general.
Because $|H|=|D|^2$ also $|K|$ has only non-negative eigenvalues.
One can now formulate the Brouwer conjecture $BC^+$ for sign-less Laplacians. 
See \cite{Cooper2021}. 

\paragraph{}
The vertices and edges of a graph form a one-dimensional simplicial 
complex with $n+m$ simplices $x$. We can look at the 
{\bf connection matrix} $L$ which is defined as $L_{xy}=1$ if and only 
if the two simplices $x,y$ intersect. In \cite{KnillEnergy2020} was noted
that the 0-1 matrix $L$ is {\bf unimodular} meaning that its determinant
is either $1$ or $-1$. More precisely, the inverse matrix $g=L^{-1}$ has entries
$w(x) w(y) \chi(U(x) \cap U(y)$, where $w(x)=(-1)^{{\rm dim}(x)}$ 
and $U(x) = \{ y, x \subset y \}$ is the {\bf star} of $x$. 
For example, if $x=y$ is a vertex, then the diagonal entry 
$g_{xx} = 1-{\rm deg}(x)$ is related to the vertex degree ${\rm deg}(x)$.
We also proved $\sum_{x,y} g(x,y) = \chi(G)$ and that the number of positive
eigenvalues of $L$ minus the number of negative eigenvalues is also $\chi(G)$.
All this could be generalized to energized complexes \cite{GreenFunctionsEnergized}.

\paragraph{}
The connection Laplacian and the Hodge Laplacians are 
defined in any dimension, for any finite abstract simplicial complex.
In one dimensions, we have the {\bf hydrogen identity} $L-L^{-1} = |H|$. 
\cite{Hydrogen}.
The eigenvalues $\lambda_j$ of $|H|$ are now related to the 
eigenvalues $\mu_j$ of $L$ by $\lambda_j = \mu_j-1/\mu_j$. 
If the eigenvalues of $|H|$ are ordered $\lambda_1 \geq \lambda_2 \geq ...$
then the eigenvalues of $L$ are ordered $\mu_1 \geq \mu_2 \geq ...$ as long as
the $\lambda_j>1$. 

\paragraph{}
We made use of the hydrogen identity to estimate the spectral radius of 
$|H|$ and so the spectral radius of $K$ from above. 
We have $\lambda_1(K) \leq max(\lambda_1(L), \lambda_1(g)$
because $\lambda_1(K) \leq \lambda_1(|K|) = \lambda_1(|H|) = 
     \lambda_1(L)-\lambda_1(L)^{-1}$. 
If $H$ be a subgraph 
of $G$ in which one of the edges has been deleted. Let $L(H)$ be the
connection Laplacian of $H$ and $L(G)$ the connection Laplacian of $G$. 

\begin{lemma}
Let $H$ be a subgraph of $G$, then 
$L(H)$ is a principle submatrix of $L(G)$. 
The eigenvalues of $L(H)$ are therefore interlaced 
with the eigenvalues of $L(G)$. 
\end{lemma} 

\begin{proof} 
Any edge can only interact with itself or with vertices. 
\end{proof} 

\paragraph{}
This generalizes to higher dimensional simplicial complexes in that 
if $H$ is a subcomplex of $G$, there is interlacing. 
It follows that the connection spectral radius of a subcomplex $H$
of $G$ is smaller or equal than the connection spectral radius of $G$. 
We see so far in experiments that the spectral radius of $g=L^{-1}$ is 
smaller or equal than the spectral radius of $L$. We do not know 
whether this is true in general. 

\section{Brouwer-Haemers}

\paragraph{}
For quivers without multiple connections, we had proven in 
\cite{Knill2024} a Brouwer-Haemers lower bound \cite{BrouwerHaemers2008} 
(written there when the eigenvalues were ordered in increasing order) 
$\lambda_j \geq d_j-(n-j)$ which 
implies in the decreasing order $\lambda_j \geq d_j-j+1$.
See also \cite{Brouwer} Proposition 3.10.2, where the decreasing 
order with the stronger $d_j-j+2$ is used, but which required some condition on the graph.)
The work of \cite{Brouwer,BrouwerHaemers2008,GroneMerrisSunder2,LiPan1999,Guo2007},
had an increased lower bound, which needed however cases like $G=K_n$. 
The slightly weaker $\lambda_j \geq d_j-j+1$ was proven with the 
snap induction proof and is true unconditionally and holds also in the quiver case
without multiple connections. 

\paragraph{}
In our earlier stage of the work we called $m + k(k+1)$ the 
{\bf Brouwer-Haemers upper bound}. For quivers with multiple connections, the adaptation 
is $m+r+k(k+1)$. Alan Lew noticed that our proof of $S_k \leq m+k(k+1)$ was incorrect. 
He also informed us of \cite{Lew2025} with the stronger $S_k \leq H_k$ with $H_k = m+k^2$
after an earlier $S_k \leq m+k^2+15k \log(k) +65k$ \cite{Lew2024}. 
We call the $H_k$ and its adaptation for quivers with redundancy $H_k=m+r+k^2$ 
now the {\bf Lew bound}. 

\paragraph{}
When looking at general graphs with loops 
but without multiple connections, the edge degree bound was 
$$  \lambda_j \leq d_j + d_{j+1} \leq 2d_j  \; . $$
This implies $S_k \leq 2 D_k$ given in \cite{Knill2024}.
Experiments show that this in general competes with the Brouwer estimate. 
There is in general a middle interval of $k$ values, where one of the two estimates is better. 
In \cite{Knill2024} we had opted to arrange the 
eigenvalues in increasing order as in Riemannian geometry.
In the context of the Brouwer conjecture most of the literature take a decreasing order. 

\paragraph{}
If $\overline{\lambda}_j$ are the eigenvalues of the graph complement 
$\overline{G}$, then $\lambda_k+ \overline{\lambda}_{n-k+1}=n$ and
$\sum_{j=1}^k \lambda_j 
+ \sum_{j=1}^{k} \overline{\lambda}_{n-j+1} =nk$ so
that the Brouwer estimate also gives a bound on the 
smallest $k$ eigenvalues of the graph complement. See \cite{Chen2019}.

\paragraph{}
In general, the vertex degree estimate $S_k \leq 2 D_k$
gives better estimates for small or large energies, 
while the Brouwer estimate $S_k \leq B_k$
estimate is better in the middle of the spectrum. 
It is worth mentioning that 
the sequence list $k \to m+k(k+1)/2$ is concave up, while the lists of 
$k \to D_k = \sum_{j=1}^k d_j$ and 
$k \to S_k(G)= \sum_{j=1}^k \lambda_j$ 
are both concave down.  If the $c S_k$ for some $c \in (1,2)$ is tangent 
to $B_k$, then the estimate $S_k \leq B_k$ holds for all $k$!

\paragraph{}
If $m \to \infty$, keeping the loop numbers as equal as possible, 
the function $k \to m+k(k+1)/2$ is close to constant 
while $k \to \sum_{j=1}^k d_j+d_{j+1}$ is close to linear. 
We can so construct examples of quivers on $n$ vertices 
for which the degree bound is smaller on half of the spectrum 
while the edge size bound is smaller on the other half of 
the spectrum. 

\paragraph{}
For $S_k$ we have now relations of the vertex degree sequences $D_k$ 
and the Brouwer sequence $B_k=m+k(k+1)/2$ which when adapted to quivers with
multiple connections becomes $B_k = m+r+k(k+1)/2$. We also adapt the Lew bound 
$H_k = m+r+k^2$. 
Let us summarize a few inequalities including the more elementary 
$D_k \leq B_{k-1}$. 

\begin{thm}
Let $G$ be an arbitrary quiver. \\
a) $D_k \leq S_k$. \\
b) $S_k \leq 2 D_k$. \\
c) $D_k \leq B_{k-1}$.  \\
d) $S_k \leq H_k$. \\
\end{thm}
\begin{proof}
a) $D_k \leq S_k$ is the Schur inequality (see e.g. \cite{Brouwer,BrouwerHaemers2008}).\\
b) Follows from $\lambda_j \leq d_j + d_{j+1}$ with the assumption $d_{n+1}=0$ (see 
   Theorem 1 in \cite{Knill2024}, a result for arbitrary quivers). \\
c) Assume first that $G$ has no multiple connections. 
   Take the sub-graph $H$ of $G$ which is the union of the $k$ largest embedded star graphs in $G$.
   There are maximally $k(k-1)$ connections between the $k$ centers of these star graphs
   so that $H$ has at least
   $D_k-k(k-1)/2$ edges. As this is $\leq m$, we have $D_k \leq m+k(k-1)/2 = B_{k-1}$. \\
   We finished the case without multiple connections. 
   If any  multiple connection is introduced, then right hand side increases by $2$ while
   the left hand can only increase maximally by $2$. \\
d) The result for graphs is shown in \cite{Lew2025}. This upgrades to quivers. For each 
   $1 \leq k \leq n$, adding a loop increases the left hand side by $1$ or less while
   the right hand side increases by $1$. When adding a multiple connection, the left hand 
   side increases by maximally $2$, while the right hand side increases by $2$. 
\end{proof}

\section*{Appendix: Quiver gradient}

\paragraph{}
We recall some {\bf quiver calculus} from \cite{Knill2024}.
Quivers are a class of geometric objects containing finite simple graphs. 
They are natural because the set of quivers is a {\bf pre-sheaf category} 
and so forms a {\bf topos}. The nomenclature is not always uniform. 
Quiver-graphs are sometimes called ``graphs" 
\cite{BR},``general graphs" \cite{Brualdi2004},
``pseudo graphs" \cite{Vasudev}, 
``loop multi-graphs" \cite{TrinajsticChemicalGraphTheory}
or ``r\'eseaux" \cite{VerdiereGraphSpectra}.

\paragraph{}
The calculus on graphs parallels the calculus on 
Riemannian manifolds. In particular, Hodge theory works, as maybe pointed out 
first by Eckmann \cite{Eckmann1944} (in arbitrary dimensions). 
The linear algebra approach to cohomology is technically simpler:
the {\bf $k$-th cohomology} can be identified with the space of 
{\bf k-harmonic forms}. They form the kernel of the k'th form Laplacian. 
Their dimension is the k'th Betti number $b_k$. 
For quivers, we only deal with one-dimensional structures ,
for which only the zero'th and first cohomology are relevant. 

\paragraph{}
Given a quiver $G=(V,E)$, with $E \subset V \times V$. 
By enumerating the $n \geq 1$ vertices and $m \geq 1$ edges
arbitrarily and by orienting the edges arbitrarily, we get a 
$m \times n$ matrix $F=d_0$ defined by $F_{e,v}=1$ if $e=(w,v)$ and 
$F_{e,v}=-1$ if $e=(v,w)$ and $F_{e,v}=0$ if $e=(v,v)$. This matrix 
is the {\bf quiver gradient}. It maps {\bf 0-forms}, 
functions on vertices to {\bf 1-forms}, functions on oriented edges.
The transpose matrix $d_1=F^T$ is the {\bf quiver divergence}. 
It maps $1$-forms to $0$-forms by telling a vertex
$v$ how much the total in-flux and out-flux coming from edges attached to $v$. 
The Kirchhoff matrix $K =d_1 d_0 F^T F$ is now a map on 0-forms, 
represented by a $n \times n$ matrix, while $K' = F^T F=d_0 d_1$ 
is a map on 1-form, represented by a $m \times m$ matrix. 
The matrix $K$ does not depend on the orientation of the edges. 
The matrix $K'$ does. But the spectrum of $K'$ does not. 

\paragraph{}
The nullity of $K$ is the $0$'th Betti number $b_0$. 
The nullity of $K'$ is the first Betti number $b_1$. 
The Euler characteristic of the quiver is $\chi(G)=n-m$. Euler-Poincar\'e
tells $\chi(G)=b_0-b_1$. 
A modern point of view is to combine the spaces of 0-forms and 
1-forms and write down $d_0,d_1$ as a nilpotent block matrices 
satisfying $d_0^2=0$ and $d_1^2=0$ and introduce the {\bf Dirac matrix}
$D=d_0+d_1$ which then gives the {\bf Hodge Laplacian} $L=D^2 = K \oplus K'$ 
consisting of the two blocks. If $A$ is a general linear map on forms, 
(here a $n \times m$ matrix), one can define the super trace 
${\rm str}(A) = \sum_{1 \leq k \leq n} A_{kk} - \sum_{n+1 \leq k \leq n+m} A_{kk}$. 
As in Riemannian geometry, the super trace of any power of the Hodge 
Laplacian is zero ${\rm str}(L^j)=0, j \geq 1$ while the super trace
of the identity matrix $1=L^0$ is the Euler characteristic $\chi(1)=n-m$. 
The {\bf McKean Singer equation} ${\rm str}(e^{-t L})=\chi(G)$ follows. 
Evaluating this at $t=0$ gives $m-n$ while evaluating this at $t \to \infty$
gives $b_0-b_1$. The {\bf McKean-Singer symmetry} has as a consequence 
that  $K$ and $K'$ are essentially isospectral. This is a fact exploited by 
Anderson-Morley \cite{AndersonMorley1985}. 

\section*{Appendix: Clovers and Ribbons}

\paragraph{}
A {\bf clover} is a graph with $n=1$ vertices with $m$ self-loops. 
The quiver gradient $F$ in that case is the $n \times 1$ matrix 
$F=\left[ \begin{array}{c} 1 \\ \cdots \\ 1 \end{array} \right]$ so that $K=F^TF=[m]$ and 
$K'=F F^T=\left[ \begin{array}{ccc} 1 & \cdots & 1 \\
                               \cdots & \cdots & \cdots \\ 
                               1 & \cdots & 1  \end{array} \right]$. 
The Betti numbers are $b_0={\rm dim}{\rm ker}(K))=0$ and 
$b_1={\rm dim}{\rm ker}(K')=m-1$, matching $\chi(G)=n-m=1-m = b_0-b_1$.

\paragraph{}
A {\bf ribbon} is a graph with $n=2$ vertices and $m \geq 2$ 
multiple connections and no loops. The quiver gradient is 
a $m \times 2$ matrix and $K=\left[ \begin{array}{cc} m & -m \\ -m & m \end{array} \right]$ 
with eigenvalues $\lambda_1=2m,\lambda_2 =0$. 
The unmodified Brouwer estimate fails already for $k=1$, 
we have $\lambda_1=2m > m + 1(1+1)/2=m+1$. 
But the modified estimate 
$\sum_{k=1}^k \lambda_j \leq m+r+k(k+1)/2$ holds, where $r$ is the number of redundant edges.
In the case of the ribbon with $m$ connections, we can take away $r=m-1$ redundant edges. 
The ribbon has the estimate $\sum_{k=1}^k \lambda_j \leq m+k(k+1)/2 + r$ sharp.

\begin{figure}[!htpb]
\scalebox{0.7}{\includegraphics{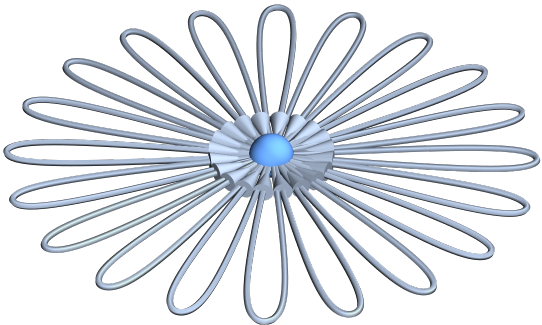}}
\scalebox{0.7}{\includegraphics{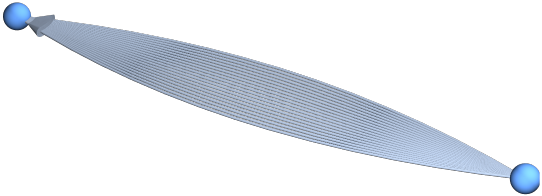}}
\label{Example}
\caption{
For a clover with $m$ loops, the Brouwer is sharp $\lambda_1=m$
For a ribbon, a graph $n=2$ vertices and $m \geq 2$ connections, 
the original Brouwer estimate is false because 
$\lambda_1=2m$ and $m+1(1+1)/2=m+1$. But the modified Brouwer estimate with 
$B_k=m+r+k(k+1)/2$ holds.
}
\end{figure}

\section*{Appendix: Hadamard perturbation}

\paragraph{}
The {\bf first Hadamard perturbation formula} $\lambda'(t) = v^T(t) E v$
describes using the eigenvector $v=v(t)$ how the eigenvalues $\lambda=\lambda(t)$ 
of a symmetric matrix $L$ 
change if it is perturbed as $L(t)=L+t E$, where $E$ is an other symmetric matrix. 
The eigenvector $v$ of $L$ to the eigenvalue $\lambda$ is assumed to have length $1$. 
Write $v(t)$ for the perturbed normalized eigenvector to the eigenvalue $\lambda(t)$. 
Take the eigenvalue equation $L(t) v(t) = \lambda(t) v(t)$  and multiply from the 
left with $v^T(t)$ to get $v^T(t) L(t) v(t) = \lambda(t) v^T(t) v(t) = \lambda(t)$. 
After differentiating using the product rule and using $L'(t)=E$ and $v \cdot v=1$ and
$Lv=\lambda v$, we end up with $\lambda'(t) = v^T(t) E v$. This derivative is bounded 
above by the norm of $E$. 

\paragraph{}
Since the first derivative stays uniformly bounded, also the higher derivatives do and
$\lambda_j(t)$ stays smooth. It has been known since Rellich's work in the 1930ies
that the eigenvalues of a smooth real symmetric matrix-valued function
can be globally labeled so that each $\lambda_j(t)$ remains smooth. 
It is not totally obvious, as the second Hadamard perturbation formula involves terms of
the form $\lambda_i-\lambda_j$ in the denominator which suggest troubles when eigenvalues
collide. This becomes indeed a problem in the non-selfadjoint case. There is also a smooth 
parametrization of the orthonormal eigenvector frame, but at eigenvalue crossings, 
a switching of labels can occur. See Chapter II of \cite{Kato}. 

\paragraph{}
In the case of the {\bf rank}-1 perturbation $E=e_l \cdot e_l^T$ 
(which is a projection onto the $l$'th basis vector,  because all matrix entries 
are $E_{ij}=0$ except for $E_{ll}=1$), this gives $\lambda'(t) = v_l(t)^2 \geq 0$. 
And because the normalization assumption was $\sum_j v_j^2=1$, 
we have $0 \leq \lambda'(t) \leq 1$. 
If a new loop is added to a quiver, then this means for the Kirchhoff
matrix that it changees from $K$ to $K + E$. 
All eigenvalues can then only increase. Because the sum of all eigenvalue changes 
is $1$ the total change of the trace from $t=0$ to $t=1$. 
We therefore also have $\sum_{j=1}^k \lambda_j' \leq 1$. The eigenvalue sum
change is between $0$ and $1$ for any $k$. 

\paragraph{}
If a new redundant edge $(a,b)$ is added, then again the eigenvalues 
can not decrease. Now, the eigenvalues and also  each of the sums 
$\sum_{j=1}^k \lambda_j$ can increase maximally by $2$. This is proven 
differently in \cite{Godsil} (Lemma 13.6.1). It can also be found in \cite{Spielman2025} 
(Corollary 6.2.2). 

This follows again from the Hadamard formula but using 
$E=(e_a-e_b) \cdot (e_a - e_b)^T$ which is twice a projection matrix. 
Hadamard perturbation gives $\lambda'(t) = v^T(t) E v(t) = (v_a-v_b)^2 \geq 0$. 
Since each eigenvalue can only increase and the sum of all eigenvalues 
changes by $2$ when moving from $K$ to $K+E$ we have the general case 
because the right hand side $m+r$ increases by $2$. 

\section*{Appendix: Illustration}

\paragraph{}
The Brouwer estimate is especially good for graphs with a small diameter.
This can be explained with being close to the complete graph. If $G$ is 
the graph complement of a graph with vertex degree $4d^2 \leq m$
then the Brouwer estimates holds. But in that regime, 
the vertex degree estimate $\lambda_j \leq d_j + d_{j-1}$ gives 
much better results. 

\begin{figure}[!htpb]
\scalebox{1.3}{\includegraphics{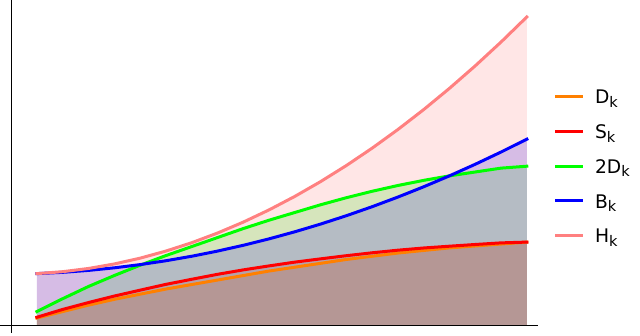}}
\label{Example}
\caption{
A random quiver without multiple connections. It has 
20 vertices, 50 edges and 30 loops. To the right, we see
the eigenvalue sum list $S_k = \sum_{j=1}^k \lambda_j$, the upper
bound using $\lambda_j \leq d_j + d_{j-1}$, the lower bound
both proven in \cite{Knill2024} and the Brouwer upper
bound. All except the Brouwer sequence $m + k(k+1)/2$ are concave down.
}
\end{figure} 

\begin{figure}[!htpb]
\scalebox{1.3}{\includegraphics{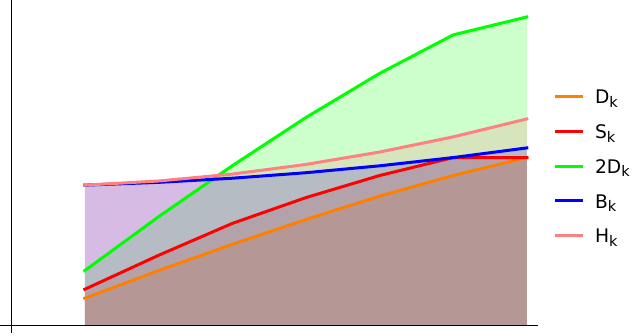}}
\label{Sharp}
\caption{
A complete graph $K_7$ with redundancy $r=40$ and $m=40+7*6/2=61$.
Since $\lambda_n=0$ for any graph without loops and for $k=n-1$
we have $122 = {\rm tr}(K) = \sum_{j=1}^{k} \lambda_j = m+k(k+1)/2+r = 61+21+40$,
the estimate is sharp here.
}
\end{figure}

\paragraph{}
The Brouwer estimate is sharp also in the quiver case if we take a complete 
graph with multiple ribbon connections. The picture shows the case with 
$$ K = \left[ \begin{array}{ccccccc}
 14 & -1 & -3 & -3 & -1 & -3 & -3 \\
 -1 & 19 & -3 & -2 & -5 & -5 & -3 \\
 -3 & -3 & 20 & -4 & -3 & -3 & -4 \\
 -3 & -2 & -4 & 13 & -1 & -1 & -2 \\
 -1 & -5 & -3 & -1 & 17 & -3 & -4 \\
 -3 & -5 & -3 & -1 & -3 & 19 & -4 \\
 -3 & -3 & -4 & -2 & -4 & -4 & 20 \\
\end{array} \right] \; . $$
It illustrates also that the Brouwer bound is for graphs with large degrees 
better in general than the upper bounds in terms of the vertex degrees. 
The edge degree estimate is better if the graph is large and the degree remains
small, like for discrete manifolds.

\section*{Appendix: Code} 

\paragraph{}
We conclude with an adaptation of the code given in \cite{Knill2024} but 
showing also the Brouwer upper bound $B_k= m+k(k+1)/2$ as well as the 
Lew upper bound $H_k = m+k^2$. It illustrates the known 
inequalities $D_k \leq S_k \leq 2D_k$ and $D_k \leq B_k$ and $S_k \leq H_k$
\cite{Lew2025}. The code was used to produce the graph in Figure 2:

\begin{tiny}
\lstset{language=Mathematica} \lstset{frameround=fttt}
\begin{lstlisting}[frame=single]
QuiverGradient[s_]:=Module[{v=VertexList[s],e=EdgeList[s],n,m,F,a,b},
 n=Length[v];m=Length[e];    F=Table[0,{m},{n}];  Q[x_]:=Subscript[x,"k"];
 Do[{a,b}={e[[k,1]],e[[k,2]]};F[[k,a]]+=1;If[a!=b,F[[k,b]]+=-1],{k,m}];F];
RandomQuiver[{n_,m_,l_,c_}]:=Module[{G=RandomGraph[{n,m}],v,e,q={},R,A},
 v=VertexList[G];e=EdgeRules[G];R=RandomChoice;A=Append;U[u_]:=u[[1]]->u[[2]];
 Do[x=R[v];q=A[q,x->x],{l}];Do[q=A[q,U[R[e]]],{c}];Graph[v,Join[e,q]]];
n=20;r=0;G=RandomQuiver[{n,50,30,r}];F=QuiverGradient[G];K=Transpose[F].F;m=Length[F]; 
Eigen=Reverse[Sort[Eigenvalues[1.0*K]]];    d=Reverse[Sort[Table[K[[k,k]],{k,n}]]];
SignEigen=Reverse[Sort[Eigenvalues[1.0*Abs[K]]]];
DegreeUpper  = Table[d[[k]]+If[k<n,d[[k+1]],0],{k,n}];     (* which is <= 2d *)
s=Table[Sum[d[[j]],{j,k}],{k,n}];           (* Schur          *)
f=Table[Sum[Eigen[[j]],{j,k}],{k,n}];       (* Eigenvalues    *)
c=Table[Sum[SignEigen[[j]],{j,k}],{k,n}];   (* Signed         *)
g=Table[Sum[DegreeUpper[[j]],{j,k}],{k,n}]; (* Knill [4]      *)
b=Table[m+k(k+1)/2 + r,{k,n}];              (* Brouwer        *)
u=Table[m+k^2+r,{k,n}];                     (* Lew 2025       *)
ListPlot[{s,f,g,b,u,c},Filling->Bottom,PlotStyle->{Orange,Red,Green,Blue,Pink,Yellow},
  PlotLegends->{Q["D"],Q["S"],Q["2D"],Q["B"],Q["H"],Q["A"]},Joined->True,Ticks->None]
\end{lstlisting}
\end{tiny} 


\vfill

\bibliographystyle{plain}

\end{document}